\documentclass[reqno]{amsart}
\usepackage{amsmath,amsthm,amssymb,amscd}
\usepackage[shortlabels]{enumitem}
\usepackage{youngtab}
\usepackage{xcolor}
\usepackage{rotating}
\usepackage{appendix}
\usepackage[numbers]{natbib}
\usepackage{subcaption}
\usepackage{tikz}

\newtheorem{theorem}{Theorem}[section]
\newtheorem{lemma}[theorem]{Lemma}

\newtheorem{conj}[theorem]{Conjecture}
\theoremstyle{definition}

\numberwithin{equation}{section}
\allowdisplaybreaks
\begin{document}
	
	%
	%
	%
	%
	%
	%
	%
	%
	%
	
	\author[Gurinder Singh]{Gurinder Singh}
	\address{Department of Mathematics, Indian Institute of Technology Guwahati, Assam, India, PIN- 781039}
	\email{gurinder.singh@iitg.ac.in}
	
	\author[Rupam Barman]{Rupam Barman}
	\address{Department of Mathematics, Indian Institute of Technology Guwahati, Assam, India, PIN- 781039}
	\email{rupam@iitg.ac.in}
	
	\title[Hook length inequalities for $t$-regular partitions in the $t$-aspect]
	{Hook length inequalities for $t$-regular partitions in the $t$-aspect}

	\date{\today}
	
	
\subjclass[2010]{11P81, 05A17, 05A19, 05A15}

\keywords{hook lengths; $t$-regular partitions; partition inequalities}

\begin{abstract}  Let $t\geq2$ and $k\geq1$ be integers. A $t$-regular partition of a positive integer $n$ is a partition of $n$ such that none of its parts is divisible by $t$. Let $b_{t,k}(n)$ denote the number of hooks of length $k$ in all the $t$-regular partitions of $n$. In this article, we prove some inequalities for $b_{t,k}(n)$ for fixed values of $k$. We prove that for any $t\geq2$, $b_{t+1,1}(n)\geq b_{t,1}(n)$, for all $n\geq0$. We also prove that $b_{3,2}(n)\geq b_{2,2}(n)$ for all $n>3$, and $b_{3,3}(n)\geq b_{2,3}(n)$ for all $n\geq0$. Finally, we state some problems for future works.
\end{abstract}

\maketitle
\section{Introduction and statement of results} 
A partition of a positive integer $n$ is a finite sequence of non-increasing positive integers $\lambda=(\lambda_1, \lambda_2, \ldots, \lambda_r)$ such that $\lambda_1+\lambda_2+\cdots +\lambda_r=n$. A \textit{Young diagram} of a partition $(\lambda_1, \lambda_2, \ldots, \lambda_r)$ is a left-justified array of boxes with the $i$-th row (from the top) having $\lambda_i$ boxes. 
For example, the Young diagram of the partition $(5,4,3,2,1)$ is shown in Figure \ref{Figure7.01} (left). The \textit{hook length} of a box in a Young diagram is the sum of the number of the boxes directly right to it, the number of boxes directly below it and 1 (for the box itself).
For example, see Figure \ref{Figure7.01} (right) for the hook lengths of each box in the Young diagram of the partition $(5,4,3,2,1)$.
\begin{figure}[h]
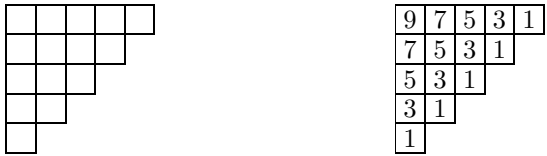

	\centering
	\begin{minipage}[b]{0.4\textwidth}
		\[\young(~~~~~,~~~~,~~~,~~,~)\]
	\end{minipage}
	\begin{minipage}[b]{0.4\textwidth}
		\[\young(97531,7531,531,31,1)\]
	\end{minipage}
	\caption{The Young diagram of the partition $(5,4,3,2,1)$ and its hook lengths}\label{Figure7.01}
\end{figure}
\par 
Hook lengths of partitions have important connections with representation theory of the symmetric groups $S_n$ and $\text{GL}_n(\mathbb{C})$. Hook lengths also appear in the Seiberg-Witten theory of random partitions, which gives the Nekrasov-Okounkov formula for arbitrary powers of Euler's infinite product in terms of hook numbers. For more details, see e.g. \cite{Garvan_1990, James, Littlewood,Nekrasov, Stanley}. Other than the ordinary partition function, hook lengths have also been studied for several restricted partition functions, for example, partitions into odd parts, partitions into distinct parts, partitions into odd and distinct parts, 
self conjugate partitions and doubled distinct partitions, see e.g. \cite{Ono_Singh, Ballantine_2023,Craig,Han_2016,Han_2017,Petreolle,Singh_Barman}.
\par 
Let $t\geq 2$ be a fixed positive integer. A $t$-regular partition of a positive integer $n$ is a partition of $n$ such that none of its parts is divisible by $t$. A $t$-distinct partition of a positive integer $n$ is a partition of $n$ such that any of its parts can occur at most $t-1$ times. For integers $t\geq2$ and $k\geq1$, let $b_{t,k}(n)$ denote the number of hooks of length $k$ in all the $t$-regular partitions of $n$ and $d_{t,k}(n)$ denote the number of hooks of length $k$ in all the $t$-distinct partitions of $n$. In \cite{Ballantine_2023}, Ballantine et al. studied hook lengths in $2$-regular partitions and $2$-distinct partitions. The authors, in \cite{Ballantine_2023}, proved that $d_{2,1}(n)\geq b_{2,1}(n)$, for all $n\geq0$. They conjectured \cite[Conjecture 1.7]{Ballantine_2023} that for every $k\geq2$, there exists an integer $N_k$ such that $b_{2,k}(n)\geq d_{2,k}(n)$, for all $n\geq N_k$. Ballantine et al. \cite[Theorem 1.8]{Ballantine_2023} proved the conjecture for $k=2,3$ and very recently Craig et al. \cite{Craig} proved it for all $k$. This type of partition inequalities between the number of hook lengths are also called hook length biases. In \cite{Singh_Barman}, we studied the hook length biases for $2$- and $3$-regular partitions for different hook lengths. We established two hook length biases for $2$-regular partitions, namely, $b_{2,2}(n)\geq b_{2,1}(n)$, for all $n>4$ and $b_{2,2}(n)\geq b_{2,3}(n)$, for all $n\geq0$. We also proposed two conjectures on biases among $2$- and $3$-regular  partitions, see \cite[Conjectures 1.6 and 6.1]{Singh_Barman}.
\par In this article, we study biases among $b_{t,k}(n)$ for fixed $k$. Our first result proves that the number of hooks of length 1 in ($t+1)$-regular partitions of any nonnegative integer $n$ is greater than or equal to the number of hooks of length 1 in $t$-regular partitions of $n$. More precisely, we have the following theorem.
\begin{theorem}\label{Theorem7.01}
Let $t\geq2$ be an integer. We have $b_{t+1,1}(n)\geq b_{t,1}(n)$, for all $n\geq0$.
\end{theorem}
For the number of hooks of length 2, we expect the same trend in $t$-regular partitions of any positive integer $n$. Our second result confirms the bias for the number of hooks of length 2 between $2$- and $3$-regular partitions. 
\begin{theorem}\label{Thm_main7.1}
For all integers $n>3$, we have $b_{3,2}(n)\geq b_{2,2}(n)$.
\end{theorem}
We observe similar inequality for hooks of length 3. In particular, we have the following theorem.
\begin{theorem}\label{Thm_main7.2}
For all nonnegative integers $n$, we have $b_{3,3}(n)\geq b_{2,3}(n)$.
\end{theorem}
\section{Proof of Theorem \ref{Theorem7.01}}
We introduce some notations. Let $\overline{\ell}(\lambda)$ denote the number of distinct parts in a partition $\lambda$. Let $h_k(\lambda)$ denote the number of hooks of length $k$ in the Young diagram of a partition $\lambda$. We recall another form of representation of a partition $\lambda$ given by
$$\lambda=(\lambda^{m_1}_1,\lambda^{m_2}_2,\ldots,\lambda^{m_r}_r),$$
where $m_i$ is the multiplicity of the part $\lambda_i$ and $\lambda_1>\lambda_2>\cdots>\lambda_r$. With this notation, for any partition $\lambda$, we consider $\lambda_{\overline{\ell}(\lambda)+1}=0$.
\begin{center}
	\begin{table}
		\caption{$\Phi_{t,n}$ for $t=3$ and $n=12$}\label{Table7.001}
		\begin{tabular}[h]{|c|c||c|c|}
			\hline
			$\tau\in\mathcal{B}_3(12)$ & $\Phi_{3,12}(\tau)\in\mathcal{B}_{4}(12)$ & $\tau\in\mathcal{B}_3(12)$ & $\Phi_{3,12}(\tau)\in\mathcal{B}_{4}(12)$ \\
			\hline 
			$(8,4)$ & $(6,3,2,1)$ & $(8,2^2)$ & $(6,2^3)$\\
			\hline 
			$(8,2,1^2)$ & $(6,2^2,1^2)$ & $(8,1^4)$ & $(6,2,1^4)$\\
			\hline 
			$(7,4,1)$ & $(7,3,1^2)$ & $(5,4,2,1)$ & $(5,3,2,1^2)$\\
			\hline 
			$(5,4,1^3)$ & $(5,3,1^4)$ & $(4^3)$ & $(3^3,1^3)$\\
			\hline 
			$(4^2,2^2)$ & $(3^2,2^2,1^2)$ & $(4^2,2,1^2)$ & $(3^2,2,1^4)$\\
			\hline 
			$(4^2,1^4)$ & $(3^2,1^6)$ & $(4,2^4)$ & $(3,2^4,1)$\\
			\hline 
			$(4,2^3,1^2)$ & $(3,2^3,1^3)$ & $(4,2^2,1^4)$ & $(3,2^2,1^5)$\\
			\hline 
			$(4,2,1^6)$ & $(3,2,1^7)$ & $(4,1^8)$ & $(3,1^9)$\\
			\hline
		\end{tabular}
	\end{table}
\end{center}
\par To prove Theorem \ref{Theorem7.01} we first prove the following lemma. Let $b_t(n)$ denote the number of $t$-regular partitions of a positive integer $n$.
\begin{lemma}\label{Lemma7.1}
Let $t\geq2$ be an integer. We have $b_{t+1}(n)\geq b_{t}(n)$, for all $n\geq0$.
\end{lemma}
\begin{proof}
Let $\mathcal{B}_t(n)$ denote the set of all $t$-regular partitions of $n$. For fixed $t$ and $n$, define a map $\Phi_{t,n}:\mathcal{B}_t(n)\rightarrow\mathcal{B}_{t+1}(n)$. For any $\tau\in\mathcal{B}_t(n)$, $\Phi_{t,n}(\tau)$ is a partition in $\mathcal{B}_{t+1}(n)$ with parts from $\tau$ which are multiple of $t+1$ changed in such a way that they are not multiple of $t+1$ and other parts remain same. A part of $\tau$ which is a multiple of $t+1$, is of the form $(t+1)(t\ell+r)=t(t+1)\ell+r(t+1)$, for some nonnegative integer $\ell$ and $1\leq r\leq t-1$ ($r\neq0$, since $\tau\in\mathcal{B}_t(n)$). Under the map $\Phi_{t,n}$ the part of part size $t(t+1)\ell+r(t+1)$ of $\tau$ is changed to $(t(t+1)\ell+rt,~r)$, which means that $t(t+1)\ell+rt$ and $r$ are considered as two parts in $\Phi_{t,n}(\tau)$. For example, Table \ref{Table7.001} shows the mapping of $3$-regular partitions of 12 to $4$-regular partitions of 12 under the map $\Phi_{3,12}$. The $3$-regular partitions of 12 which are not listed in Table \ref{Table7.001} are also $4$-regular partitions of 12 and hence mapped to themselves. Next, we prove that $\Phi_{t,n}$ is an injective map. For $\tau_1,\tau_2\in \mathcal{B}_t(n)$, let $\Phi_{t,n}(\tau_1)=\Phi_{t,n}(\tau_2)$. The parts of $\Phi_{t,n}(\tau_1)$ and $\Phi_{t,n}(\tau_2)$ which are not of the type $t(t+1)\ell+rt$ or $r$ (for some nonnegative integer $\ell$ and $1\leq r\leq t-1$) are also the parts of $\tau_1$ and $\tau_2$. If $t(t+1)\ell+rt$ and $r$ are the parts of $\Phi_{t,n}(\tau_1)$ and $\Phi_{t,n}(\tau_2)$ with multiplicity, say $m$, then $(t+1)(t\ell+r)$ is a part in both $\tau_1$ and $\tau_2$ with multiplicity $m$. This implies that $\tau_1=\tau_2$. Therefore, $\Phi_{t,n}$ is an injective map. This proves that $|\mathcal{B}_t(n)|\leq |\mathcal{B}_{t+1}(n)|$, i.e., $b_{t}(n)\leq b_{t+1}(n)$.  
\end{proof}
\begin{proof}[Proof of Theorem \ref{Theorem7.01}]
It is easy to observe that for any partition $\tau$ the number of hooks of length 1 in the Young diagram of $\tau$ is same as the number of distinct parts of $\tau$, i.e., $h_1(\tau)=\overline{\ell}(\tau)$. From Lemma \ref{Lemma7.1}, we have that $b_t(n)\leq b_{t+1}(n)$, for all $n\geq0$. Note that the number of distinct parts in $\tau\in\mathcal{B}_t(n)$ is less than or equal to the number of distinct parts in $\Phi_{t,n}(\tau)\in\mathcal{B}_{t+1}(n)$. Therefore, $b_{t,1}(n)\leq b_{t+1,1}(n)$, for all $n\geq0$.
\end{proof}
\section{Proofs of Theorems \ref{Thm_main7.1} and \ref{Thm_main7.2}}
We represent a partition $\tau$ from $\mathcal{B}_2(n)$ by
$$\left((6k+5)^{\alpha_{k,5}},~(6k+3)^{\alpha_{k,3}},~(6k+1)^{\alpha_{k,1}}  \right)_{k\geq0},$$	
where 
$\alpha_{k,j}$ is the multiplicity of the part $6k+j$. From a partition $\tau\in\mathcal{B}_2(n)$, we define triples by $$\tau_k=\left((6k+5)^{\alpha_{k,5}},~(6k+3)^{\alpha_{k,3}},~(6k+1)^{\alpha_{k,1}}  \right)_k,$$ 
such that $\tau=(\tau_k)_{k\geq0}$. The map $\Phi_{2,n}: \mathcal{B}_2(n)\rightarrow\mathcal{B}_3(n)$ is defined by
\begin{align*}
\Phi_{2,n}(\tau)&=\Phi_{2,n}\left(((6k+5)^{\alpha_{k,5}},~(6k+3)^{\alpha_{k,3}},~(6k+1)^{\alpha_{k,1}})_{k\geq0} \right)\nonumber\\
&:=\left((6k+5)^{\alpha_{k,5}},~(6k+2)^{\alpha_{k,3}},~(6k+1)^{\alpha_{k,1}} ;~1^{\alpha_{k,3}}\right)_{k\geq0}.
\end{align*}
We take
$(\Phi_{2,n}(\tau))_k=\left\{
\begin{array}{lll}
\left((6k+5)^{\alpha_{k,5}},~(6k+2)^{\alpha_{k,3}},~(6k+1)^{\alpha_{k,1}} \right)_k & \text{if}\ k\geq1;\\
\left(5^{\alpha_{0,5}},~2^{\alpha_{0,3}},~1^{\alpha_{0,1}+\sum_{i\geq0}\alpha_{i,3}} \right) & \text{if}\ k=0.
\end{array}
\right.$
\subsection{Proof of Theorem \ref{Thm_main7.1}}
In the Young diagram of a partition, a hook of length 2, which we call a 2-hook, may arise in two different ways. 
\begin{enumerate}[(a)]
\item We call a 2-hook an $m$-2-hook if it appears due to the multiplicity of a part being greater than one. \item We call a 2-hook a $g$-2-hook if it appears in the column corresponding to a part  $\lambda_i$ with gap between $\lambda_i$ and $\lambda_{i+1}$ being more than 1. 
\end{enumerate}
For example, see Figure \ref{Figure7.2}. 
\begin{figure}[h]
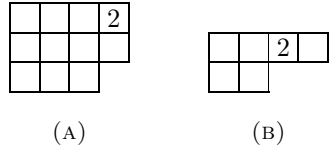

	\centering
	\begin{minipage}[b]{0.2\textwidth}
		\[\young(~~~2,~~~~,~~~)\]
		\subcaption{}
	\end{minipage}
	\begin{minipage}[b]{0.2\textwidth}
		\[\young(~~2~,~~)\]
		\subcaption{}
	\end{minipage}
	\caption{Types of 2-hooks: (a) $m$-2-hook and (b) $g$-2-hook}\label{Figure7.2}
\end{figure}
\begin{proof}[Proof of Theorem \ref{Thm_main7.1}]
Note that for $k\geq1$, $\tau_k$ and $(\Phi_{2,n}(\tau))_k$ have the same number of $m$-2-hooks but the number of $g$-2-hooks for $\tau_k$ is either equal to or one more than the number of $g$-2-hooks for $(\Phi_{2,n}(\tau))_k$.  Also, the number of 2-hooks in  $\tau_0$ and $(\Phi_{2,n}(\tau))_0$ differ by at most 1. 
\par 
The idea of our proof is as follows. The number of 2-hooks in $\tau_k$ and $(\Phi_{2,n}(\tau))_k$ differ by at most 1. For the case in which a triple $\tau_k$ loses a 2-hook while going under the map $\Phi_{2,n}$, we  assign a distinct triple to $\tau_k$ to compensate the loss of one 2-hook for it. For the other case, when the number of 2-hooks is same for $\tau_k$ and $(\Phi_{2,n}(\tau))_k$, we are done. In this way, we prove that a partition $\tau\in \mathcal{B}_2(n)$ either has the number of 2-hooks less than the number of 2-hooks in $\Phi_{2,n}(\tau)\in\mathcal{B}_3(n)$, or (in the other case, when $\tau$ loses 2-hooks while going under $\Phi_{2,n}$) along with $\Phi_{2,n}(\tau)$ we associate a partition, say $\tau'$, to $\tau$ which compensates the loss.
\par We study triples $\tau_k$ in four cases. The cases in which $(\Phi_{2,n}(\tau))_k$ has one 2-hook fewer than $\tau_k$, we associate a 4-tuple (a part of a partition in $\mathcal{B}_3(n)$ and different than $(\Phi_{2,n}(\tau))_k$) to $\tau_k$, which has at least one 2-hook.\\
\textbf{Case 1: $\alpha_{k,3}=0$.} In this case, the number of 2-hooks in $\tau_k$ is the same as the number of 2-hooks in $(\Phi_{2,n}(\tau))_k=\tau_k$, if $k\geq1$. For $k=0$, the number of 2-hooks in $(\Phi_{2,n}(\tau))_0$ is greater than or equal to the number of 2-hooks in $\tau_0$. \\
\textbf{Case 2: $\alpha_{k,1}=0$.} For $k\geq1$, the number of 2-hooks in $\tau_k$ is the same as the number of 2-hooks in $(\Phi_{2,n}(\tau))_k$. For $k=0$, if $\tau_0\ne(5^{\alpha_{0,5}},~3)$ then the number of 2-hooks in $\tau_k$ is less than or equal to the number of 2-hooks in $(\Phi_{2,n}(\tau))_0$. If $\tau_0=(5^{\alpha_{0,5}},~3)$ and $\alpha_{0,5}\ne0$ then we cover the loss of a 2-hook by associating $\rho_0:=(5^{\alpha_{0,5}-1},4^2,1^x)$ to $\tau_0$, where $x$ is the multiplicity of 1 coming in the scene due to other triples of $\tau=(\tau_k)_{k\geq0}$.
If $\tau_0=(3)$ (i.e., $\alpha_{0,5}=0$ in $\tau_0=(5^{\alpha_{0,5}},~3)$) then we cover the loss of 2-hook as follows. Since $n>3$, there is the smallest part with part size greater than or equal to 5, say $\lambda_i$. In this case, we take 5 from the part $\lambda_i$ and associate $(4^2)$ to $\tau_0=(3)$. For the remaining part $\lambda_i-5$, we proceed by considering it as a part of the partition under consideration and if $\lambda_i-5=6r+6$, for some $r\geq0$, then we change it to $(6r+5,1)$ along with other parts while applying $\Phi_{2,n}$. In this case $\rho_0:=(5^w,4^{2+z},2^y,1^x)$, where $x$ is the multiplicity of 1 coming due to the other triples; $y,z,w$ are the multiplicities of parts $2,4,5$ (respectively), which may occur due to the part $\lambda_i-5$. For example, if $\tau=(11,3)$ then $\rho_0=(5,4^2,1)$; if $\tau=(7,3)$ then $\rho_0=(4^2,2)$; if $\tau=(9,3)$ then $\rho_0=(4^3)$; if $\tau=(5^2,3)$ then $\rho_0=(5,4^2)$.\\
\textbf{Case 3: $\alpha_{k,3}>1$ and $\alpha_{k,1}\neq0$.} In this case, there is at most one loss of 2-hook in $(\Phi_{2,n}(\tau))_k$, which we cover by the following map
\begin{align*}
f(\tau_k)&=f\left(((6k+5)^{\alpha_{k,5}},~(6k+3)^{\alpha_{k,3}},~(6k+1)^{\alpha_{k,1}}  )_{k}\right)\nonumber\\
&=\left((6k+5)^{\alpha_{k,5}},~6k+4,~(6k+2)^{\alpha_{k,3}-1},~(6k+1)^{\alpha_{k,1}} ;~1^{\alpha_{k,3}-2}\right)_{k}.
\end{align*}
In this case, we associate 
$$\sigma_k:=\left((6k+5)^{\alpha_{k,5}},~6k+4,~(6k+2)^{\alpha_{k,3}-1},~(6k+1)^{\alpha_{k,1}} \right)_{k}$$
to $\tau_k$ for $k\geq1$. For $\tau_0$, $\sigma_0=(5^{\alpha_{0,5}},~4,~2^{\alpha_{0,3}-1},~1^{\alpha_{0,1}+s})$, where $s$ is the number of 1s due to other triples.\\
\textbf{Case 4: $\alpha_{k,3}=1$ and $\alpha_{k,1}\neq0$.} In this case also, there is at most one loss of 2-hook in $(\Phi_{2,n}(\tau))_k$, which we cover by the following map
\begin{align*}
g(\tau_k)&=g\left(((6k+5)^{\alpha_{k,5}},~6k+3,~(6k+1)^{\alpha_{k,1}})_{k} \right)\nonumber\\
&=\left\{
\begin{array}{lll}
\left((6k+5)^{\alpha_{k,5}},~6k+4,~(6k+1)^{\alpha_{k,1}-1};~6k-1,~1\right)_{k} & \text{if}\ k\geq1;\\
\left(5^{\alpha_{0,5}},~4,~1^{\alpha_{0,1}-1}\right) & \text{if}\ k=0.
\end{array}
\right.
\end{align*}
Here, for $k\geq1$, part $6k-1=6(k-1)+5$ is considered as a part of $\tau_{k-1}$, doing which does not decrease the number of 2-hooks in $\tau_{k-1}$.
In this case, we associate 
$$\delta_k:=\left((6k+5)^{\alpha_{k,5}},~6k+4,~(6k+1)^{\alpha_{k,1}-1}\right)_{k}$$
to $\tau_k$ for $k\geq1$. For $\tau_0$, $\sigma_0=(5^{\alpha_{0,5}},~4,~1^{\alpha_{0,1}-1+s})$, where $s$ is the number of 1s due to other triples.
\par Now, let $\tau=(\tau_k)_{k\geq0}\in\mathcal{B}_{2}(n)$. We consider the following two cases.\\
\textbf{Case A.} If the number of 2-hooks in $\tau_k$ is less than or equal to the number of 2-hooks in $(\Phi_{2,n}(\tau))_k$ for all $k$ (from Case 1 and Case 2), then we define $\tau^{*}:=\Phi_{2,n}(\tau)$. Clearly, $h_2(\tau)\leq h_2(\tau^{*})$.\\
\textbf{Case B.} If for any $k\geq0$, the number of 2-hooks in $(\Phi_{2,n}(\tau))_k$ is one less than the number of 2-hooks in $\tau_k$, we take $\tau'$ to be a partition in $\mathcal{B}_3(n)$ with $(\Phi_{2,n}(\tau))_k$ replaced by the required $\rho_0$, $\sigma_k$ or $\delta_k$, which has at least one 2-hook. In this case, we define $\tau^{*}:=(\Phi_{2,n}(\tau),\tau')$ and $h_2(\tau^{*}):=h_2(\Phi_{2,n}(\tau))+h_2(\tau')$ (Note that $\tau^{*}$ is a set of two partitions from $\mathcal{B}_3(n)$). In that way, in this case also we have, $h_2(\tau)\leq h_2(\tau^{*})$.
\par Finally, since $\Phi_{2,n}$ is an injective map, all $\Phi_{2,n}(\tau)$ are distinct. Note that $(\Phi_{2,n}(\tau))_k$, $\rho_0$, $\sigma_k$ and $\delta_k$ are all distinct as well. Therefore, $\tau'$ and $\Phi_{2,n}(\tau)$ are also distinct.  For example, see Table \ref{Table7.1}.
\par Hence, we have
\begin{align*}
b_{2,2}(n)&=\sum_{\tau\in\mathcal{B}_2(n)}h_2(\tau)=\sum_{\substack{\tau\in\mathcal{B}_2(n)\\ \text{Case A}}}h_2(\tau)+\sum_{\substack{\tau\in\mathcal{B}_2(n)\\ \text{Case B}}}h_2(\tau)\\
&\leq \sum_{\substack{\tau\in\mathcal{B}_2(n)\\ \text{Case A}}}h_2(\Phi_{2,n}(\tau))+\sum_{\substack{\tau\in\mathcal{B}_2(n)\\ \text{Case B}}}(h_2(\Phi_{2,n}(\tau)+h_2(\tau'))\\
&=\sum_{\tau\in\mathcal{B}_2(n)}h_2(\tau^{*})\leq \sum_{\tau\in\mathcal{B}_3(n)}h_2(\tau)=b_{3,2}(n).
\end{align*}
This completes the proof of the theorem. 
\end{proof}
\begin{center}
\begin{table}
\caption{Outline of the proof of Theorem \ref{Thm_main7.1} for $n=13$}\label{Table7.1}
\begin{tabular}[h]{|c|c|c|c|c|}
\hline
$\tau\in\mathcal{B}_2(13)$ & $\tau^{*}=\Phi_{2,n}(\tau)$ & $\tau^{*}=(\Phi_{2,n}(\tau),\tau')$ & $h_2(\tau)$ & $h_2(\tau^{*})$\\
\hline 
$(13)$ & $(13)$ & & 1 & 1 \\
\hline
$(11,1^2)$ & $(11,1^2)$ & & 2 & 2\\
\hline
$(9,3,1)$ &  & $((8,2,1^3),(8,4,1))$ & 2 & 2+2\\
\hline
$(9,1^4)$ & $(8,1^5)$ & & 2 & 2\\
\hline
$(7,5,1)$ & $(7,5,1)$ & & 2 & 2\\
\hline
$(7,3^2)$ & $(7,2^2,1^2)$ & & 3 & 3\\
\hline
$(7,3,1^3)$ &  & $((7,2,1^4),(7,4,1^2))$ & 3 & 2+3\\
\hline
$(7,1^6)$ & $(7,1^6)$ & & 2 & 2\\
\hline
$(5^2,3)$ &  & $((5^2,2,1),(5,4^2))$ & 3 & 2+2\\
\hline 
$(5^2,1^3)$ & $(5^2,1^3)$ & & 3 & 3\\
\hline 
$(5,3^2,1^2)$ &  & $((5,2^2,1^4),(5,4,2,1^2))$ & 4 & 3+2\\
\hline 
$(5,3,1^5)$ &  & $((5,2,1^6),(5,4,1^4))$ & 3 & 2+2\\
\hline 
$(5,1^8)$ & $(5,1^8)$ & & 2 & 2\\
\hline
$(3^4,1)$ &  & $((2^4,1^5),(4,2^3,1^3))$ & 2 & 2+3\\
\hline 
$(3^3,1^4)$ &  & $((2^3,1^7),(4,2^2,1^5))$ & 3 & 2+3\\
\hline 
$(3^2,1^7)$ &  & $((2^2,1^9),(4,2,1^7))$ & 3 & 2+2\\
\hline 
$(3,1^{10})$ &  & $((2,1^{11}),(4,1^9))$ & 2 & 1+2\\
\hline 
$(1^{13})$ & $(1^{13})$ & & 1 &1\\ 
\hline
\hline 
\multicolumn{3}{|c|}{Total number of 2-hooks}
 & 43 & 57 \\ 
\hline
\end{tabular}
\end{table}
\end{center}
\subsection{Proof of Theorem \ref{Thm_main7.2}}
In the Young diagram of a partition, a hook of length 3, which we call a 3-hook may arise in four different ways. 
\begin{enumerate}[(a)]
\item We call a 3-hook an $m_3$-3-hook if it arises due to the multiplicity of a part being greater than two and it appears in the third last column from the columns corresponding to $\lambda_i$ in the Young diagram.
\item We call a 3-hook a $g$-3-hook if it appears in the column corresponding to a part $\lambda_i$ with gap between $\lambda_i$ and $\lambda_{i+1}$ being more than 2.
\item We call a 3-hook an $m_2$-3-hook if it arises due to the multiplicity of a part $\lambda_i$ being at least two and it appears in the second last column from the columns corresponding to $\lambda_i$ in the Young diagram.
\item We call a 3-hook a $s$-3-hook if it appears in the column corresponding to a part $\lambda_i$ with gap between $\lambda_i$ and $\lambda_{i+1}$ being exactly 1 and the part $\lambda_{i+1}$ occurs once.
\end{enumerate}
For example, see Figure \ref{Figure7.1}.
\begin{figure}[h]
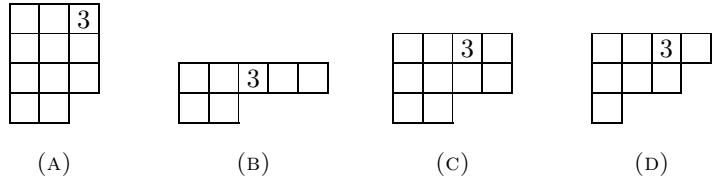

	\centering
	\begin{minipage}[b]{0.2\textwidth}
		\[\young(~~3,~~~,~~~,~~)\]
		\subcaption{}
	\end{minipage}
	\begin{minipage}[b]{0.2\textwidth}
		\[\young(~~3~~,~~)\]
		\subcaption{}
	\end{minipage}
	\begin{minipage}[b]{0.2\textwidth}
		\[\young(~~3~,~~~~,~~)\]
		\subcaption{}
	\end{minipage}
	\begin{minipage}[b]{0.2\textwidth}
		\[\young(~~3~,~~~,~)\]
		\subcaption{}
	\end{minipage}
	\caption{Types of 3-hooks: (a) $m_3$-3-hook, (b) $g$-3-hook, (c) $m_2$-3-hook, and (d) $s$-3-hook}\label{Figure7.1}
\end{figure}
\begin{proof}[Proof of Theorem \ref{Thm_main7.2}]
Similar to the case of 2-hooks, for $k\geq1$, $\tau_k$ and $(\Phi_{2,n}(\tau))_k$ have same number of $m_3$-3-hooks. Also, the number of $g$-3-hooks for $\tau_k$ is same as the number of $g$-3-hooks for $(\Phi_{2,n}(\tau))_k$, when $k\geq1$. However, the number of $m_2$-3-hooks for $(\Phi_{2,n}(\tau))_k$ is either equal to or one less than the number of $m_2$-3-hooks for $\tau_k$, for $k\geq1$. Note that for a 2-regular partition, there is no $s$-3-hook in its Young diagram. Therefore, the number of 3-hooks in $\tau_k$ can be, at the most, one less than the number of 3-hooks in $(\Phi_{2,n}(\tau))_k$. For $k=0$, the number of $m_3$-3-hooks for $\tau_0$ is either equal to or one less than the number of $m_3$-3-hooks for $(\Phi_{2,n}(\tau))_0$. Whereas, the number of $g$-3-hooks for $\tau_0$ is either equal to or one more than the number of $g$-3-hooks for $(\Phi_{2,n}(\tau))_0$ and same is the case for $m_2$-3-hooks. Therefore, the number of 3-hooks in $\tau_0$ can be, at the most, two less than the number of 3-hooks in $(\Phi_{2,n}(\tau))_0$.
\par The idea of the proof is similar to the proof of Theorem \ref{Thm_main7.1}. A partition $\tau\in \mathcal{B}_2(n)$ either has the number of 3-hooks less than or equal to the number of 3-hooks in $\Phi_{2,n}(\tau)\in\mathcal{B}_3(n)$, or (in the other case, when $\tau$ loses 3-hooks while going under $\Phi_{2,n}$) we associate a different partition, say $\tau'$, to $\tau$ which compensates the loss.
\par We study the triples $\tau_k$ in two cases. \\
\textbf{Case 1: $k\geq1$.} Note that the number of $m_2$-3-hooks for $\tau_k$ decreases under the map $\Phi_{2,n}$ only when $\alpha_{k,3}\geq2$ and $\alpha_{k,1}\geq1$. In that case we associate a new tuple to $\tau_k$ to cover the loss of an $m_2$-3-hook by using the following map
\begin{align*}
F(\tau_k)&=F\left(((6k+5)^{\alpha_{k,5}},~(6k+3)^{\alpha_{k,3}},~(6k+1)^{\alpha_{k,1}})_{k}\right)\nonumber\\
&=\left((6k+5)^{\alpha_{k,5}},~(6k+4)^2,~(6k+2)^{\alpha_{k,3}-2},~(6k+1)^{\alpha_{k,1}-1} ;~(6k-1),~1^{\alpha_{k,3}-2}\right)_{k}.
\end{align*}
In this case, we associate 
$$\theta_k:=\left((6k+5)^{\alpha_{k,5}},~(6k+4)^2,~(6k+2)^{\alpha_{k,3}-2},~(6k+1)^{\alpha_{k,1}-1} \right)_{k}$$
to $\tau_k$, which clearly has at least one 3-hook ($m_2$-3-hook corresponding to the parts $6k+4$) to compensate the loss. Here, part $6k-1=6(k-1)+5$ is considered as a part of $\tau_{k-1}$, doing which does not decrease the number of 2-hooks in $\tau_{k-1}$.\\
\textbf{Case 2: $k=0$.} In this case, there might be loss of at most two 3-hooks and that also when $\alpha_{0,3}>0$. We have $\tau_0=(5^{\alpha_{0,5}},~3^{\alpha_{0,3}},~1^{\alpha_{0,1}})$. Depending on the multiplicity of the part 3, $\alpha_{0,3}=4\ell+j$, $0\leq j\leq3$, we consider the following two cases.\\
\textbf{Subcase (a): $\ell>0$.} In this case, we compensate the loss with the following map:
\begin{align*}
G(\tau_0)=\left\{
\begin{array}{lll}
(5^{\alpha_{0,5}},~4^{3\ell},~1^{\alpha_{0,1}}) & \text{if}\ j=0; \\
(5^{\alpha_{0,5}},~4^{3\ell},~2,~1^{\alpha_{0,1}+1}) & \text{if}\ j=1; \\
(5^{\alpha_{0,5}},~4^{3\ell+1},~1^{\alpha_{0,1}+2}) & \text{if}\ j=2; \\
(5^{\alpha_{0,5}},~4^{3\ell+2},~1^{\alpha_{0,1}+1}) & \text{if}\ j=3. \\
\end{array}
\right.
\end{align*}
Clearly, in each case $G(\tau_0)$ has at least two 3-hooks. We associate $\theta_0$ to $\tau_0$, which is $G(\tau_0)$ including the multiplicity of part size 1 coming from the other triples $\tau_k$.\\
\textbf{Subcase (b): $\ell=0$.} Here, $j=0$ can not be the case since $\alpha_{0,3}>0$.  For $j=3$, the loss of a $3$-hook can be covered by the same map $G$ in the above subcase, i.e.,
\begin{align*}
G(\tau_0)=G((5^{\alpha_{0,5}},~3^3,~1^{\alpha_{0,1}}))=(5^{\alpha_{0,5}},~4^{2},~1^{\alpha_{0,1}+1}).
\end{align*}
We associate $\theta_0=(5^{\alpha_{0,5}},~4^{2},~1^{\alpha_{0,1}+1+\sum_{k\geq1}\alpha_{k,3}})$ to $\tau_0$ in this case.\\
For $j=1$, $(\Phi_{2,n}(\tau))_0=\Phi_{2,n}\left((5^{\alpha_{0,5}},~3,~1^{\alpha_{0,1}})\right)=\left(5^{\alpha_{0,5}},~2,~1^{\alpha_{0,1}+\sum_{k\geq0}\alpha_{k,3}} \right)$. If either $\alpha_{0,1}\neq0$ or $\sum_{k\geq0}\alpha_{k,3}\neq1$ then there is no loss of 3-hook under $\Phi_{2,n}$.  If $\alpha_{0,1}=0$ and $\sum_{k\geq0}\alpha_{k,3}=1$, then the loss of a 3-hook is covered by taking 1 from $\sum_{k\geq0}\alpha_{k,3}$ and changing part size 3 to part size 4 as follows
\begin{align*}
H(\tau_0)=H((5^{\alpha_{0,5}},~3))=\left(5^{\alpha_{0,5}},~4\right).
\end{align*}
In this case $\theta_0=H(\tau_0)$.\\
For $j=2$, $(\Phi_{2,n}(\tau))_0=\Phi_{2,n}\left((5^{\alpha_{0,5}},~3^2,~1^{\alpha_{0,1}})\right)=\left(5^{\alpha_{0,5}},~2^2,~1^{\alpha_{0,1}+\sum_{k\geq0}\alpha_{k,3}} \right)$. If either $\alpha_{0,1}\neq0$ or $\sum_{k\geq0}\alpha_{k,3}\neq0$, then the loss of a 3-hook is covered by 
\begin{align*}
I(\tau_0)=I((5^{\alpha_{0,5}},~3^2,~1^{\alpha_{0,1}}))=\left(5^{\alpha_{0,5}},~4,~1^{\alpha_{0,1}+\sum_{k\geq0}\alpha_{k,3}}\right).
\end{align*}
For $\alpha_{0,1}=0$ and $\sum_{k\geq0}\alpha_{k,3}=0$, let $n>6$. Then there is the smallest part with part size greater than or equal to 5, say $\lambda_i$. In this case, we take 4 from the part $\lambda_i$ and associate $(4,~2^3)$ to $(3^2)$. For the remaining part $\lambda_i-4$, we proceed by considering it as a part of the partition and if $\lambda_i-4=6r+6$, for some $r\geq0$, then we change it to $(6r+5,1)$ along with other parts while applying $\Phi_{2,n}$. If the final multiplicity of 1 is $v$ then $(1^v)$ is changed to $(2^{v/2})$ or $(2^{(v-1)/2},1)$, depending on $v$ being even or odd, respectively. In this case $\theta_0:=(5^w,4^{1+z},2^y,1^x)$, where $x$ is the multiplicity of 1 (either 0 or 1); $y,z,w$ are the multiplicities of parts $2,4,5$ (respectively), which may also occur due to the part $\lambda_i-4$. For example, if $\tau=(11,3^2)$ then $\tau'=(7,4,2^3)$ and $\theta_0=(4,2^3)$; if $\tau=(7,3^2)$ then $\tau'=\theta_0=(4,2^4,1)$; if $\tau=(5^2,3^2)$ then $\tau'=\theta_0=(5,4,2^3,1)$.
\par Now, let $\tau=(\tau_k)_{k\geq0}\in\mathcal{B}_{2}(n)$ and $n>6$. We consider two cases.\\
\textbf{Case A.} If the number of 3-hooks in $\tau_k$ is less than or equal to the number of 3-hooks in $(\Phi_{2,n}(\tau))_k$ for all $k$, then we define $\tau^{*}:=\Phi_{2,n}(\tau)$. Clearly, $h_3(\tau)\leq h_3(\tau^{*})$.\\
\textbf{Case B.} If for any $k\geq0$, the number of 3-hooks in $(\Phi_{2,n}(\tau))_k$ is less than the number of 3-hooks in $\tau_k$, we take $\tau'$ to be a partition in $\mathcal{B}_3(n)$ with $(\Phi_{2,n}(\tau))_k$ replaced by the required $\theta_k$, which covers the loss of one or two 3-hooks. In this case, we define $\tau^{*}:=(\Phi_{2,n}(\tau),\tau')$ and $h_3(\tau^{*}):=h_3(\Phi_{2,n}(\tau))+h_3(\tau')$ (Note that $\tau^{*}$ is a set of two partitions from $\mathcal{B}_3(n)$). In this case also we have, $h_3(\tau)\leq h_3(\tau^{*})$.
\par Since $\Phi_{2,n}$ is an injective map, all $\Phi_{2,n}(\tau)$ are distinct. Note that $(\Phi_{2,n}(\tau))_k$ and $\theta_k$ are all distinct as well. Therefore, all $\tau'$ and $\Phi_{2,n}(\tau)$ are also distinct.  For example, see Table \ref{Table7.2}. Hence, we have for $n>6$
\begin{align*}
b_{2,3}(n)&=\sum_{\tau\in\mathcal{B}_2(n)}h_3(\tau)=\sum_{\substack{\tau\in\mathcal{B}_2(n)\\ \text{Case A}}}h_3(\tau)+\sum_{\substack{\tau\in\mathcal{B}_2(n)\\ \text{Case B}}}h_3(\tau)\\
&\leq \sum_{\substack{\tau\in\mathcal{B}_2(n)\\ \text{Case A}}}h_3(\Phi_{2,n}(\tau))+\sum_{\substack{\tau\in\mathcal{B}_2(n)\\ \text{Case B}}}(h_3(\Phi_{2,n}(\tau)+h_2(\tau'))\\
&=\sum_{\tau\in\mathcal{B}_2(n)}h_3(\tau^{*})\leq \sum_{\tau\in\mathcal{B}_3(n)}h_3(\tau)=b_{3,3}(n).
\end{align*}
For $0\leq n\leq6$, it is easy to check that the inequality $b_{2,3}(n)\leq b_{3,3}(n)$ holds. This completes the proof. 
\end{proof}
\begin{center}
	\begin{table}
		\caption{Outline of the proof of Theorem \ref{Thm_main7.2} for $n=13$}\label{Table7.2}
		\begin{tabular}[h]{|c|c|c|c|c|}
			\hline
			$\tau\in\mathcal{B}_2(13)$ & $\tau^{*}=\Phi_{2,n}(\tau)$ & $\tau^{*}=(\Phi_{2,n}(\tau),\tau')$ & $h_3(\tau)$ & $h_3(\tau^{*})$\\
			\hline 
			$(13)$ & $(13)$ & & 1 & 1 \\
			\hline
			$(11,1^2)$ & $(11,1^2)$ & & 1 & 1\\
			\hline
			$(9,3,1)$ &  $(8,2,1^3)$ &  & 1 & 2\\
			\hline
			$(9,1^4)$ & $(8,1^5)$ & & 2 & 2\\
			\hline
			$(7,5,1)$ & $(7,5,1)$ & & 1 & 1\\
			\hline
			$(7,3^2)$ & &$((7,2^2,1^2),(4,2^4,1))$  & 3 & 1+2\\
			\hline
			$(7,3,1^3)$ &  $(7,2,1^4)$ & & 2 & 2\\
			\hline
			$(7,1^6)$ & $(7,1^6)$ & & 2 & 2\\
			\hline
			$(5^2,3)$ &  $(5^2,2,1)$ & &2 & 3\\
			\hline 
			$(5^2,1^3)$ & $(5^2,1^3)$ & & 3 & 3\\
			\hline 
			$(5,3^2,1^2)$ &  & $((5,2^2,1^4),(5,4,1^4))$ & 1 & 2+2\\
			\hline 
			$(5,3,1^5)$ & $(5,2,1^6)$ &  &1 & 2\\
			\hline 
			$(5,1^8)$ & $(5,1^8)$ & & 2 & 2\\
			\hline
			$(3^4,1)$ &  & $((2^4,1^5),(4^3,1))$ & 2 & 2+3\\
			\hline 
			$(3^3,1^4)$ &  & $((2^3,1^7),(4^2,1^5))$ & 3 & 2+3\\
			\hline 
			$(3^2,1^7)$ &  & $((2^2,1^9),(4,1^9))$ & 2 & 1+2\\
			\hline 
			$(3,1^{10})$ &  $(2,1^{11})$ &  & 2 & 2\\
			\hline 
			$(1^{13})$ & $(1^{13})$ & & 1 &1\\ 
			\hline
			\hline 
			\multicolumn{3}{|c|}{Total number of 3-hooks}
			& 32 & 44 \\ 
			\hline
		\end{tabular}
	\end{table}
\end{center}
\section{Concluding Remarks}
Let $t\geq2$ and $k\geq1$ be integers. The main motive of our study is to find the biases among $b_{t,k}(n)$ and $d_{t,k}(n)$, for fixed values of $k$. If $\lambda$ is a $t$-distinct partition of $n$, then it is also a $(t+1)$-distinct partition of $n$. Therefore, $d_{t+1,k}(n)\geq d_{t,k}(n)$, for all $n\geq0$.
For a fixed value of $k$, we want to find the biases in the following diagram:
\begin{center}
\begin{tikzpicture}
\node (R1) at (0,0) {$b_{t+1,k}(n)$};
\node (R2) at (0,1.5) {$b_{t,k}(n)$};	
\node (R3) at (2.5,0) {$d_{t+1,k}(n)$};
\node (R4) at (2.5,1.5) {$d_{t,k}(n)$};
\node (R5) at (0,0.75) {$?$};
\node (R6) at (1.25,0) {$?$};
\node (R6) at (1.25,1.5) {$?$};
\node (R6) at (2.5,0.75) {$\rotatebox{90}{\ensuremath{\geqslant}}$};	
\end{tikzpicture}
\end{center}
In \cite[Theorem 1.6]{Li_Wang_2019}, Li and Wang proved that for all $t\geq2$ and $n\geq0$
$$\sum_{\lambda\in\mathcal{D}_t(n)}\overline{\ell}(\lambda)-\sum_{\lambda\in\mathcal{B}_t(n)}\overline{\ell}(\lambda)\geq0,$$
where $\mathcal{D}_t(n)$ is the set of all $t$-distinct partitions of $n$. Since $h_1(\lambda)=\overline{\ell}(\lambda)$, it implies that $d_{t,1}(n)\geq b_{t,1}(n)$, for all $t\geq2$ and $n\geq0$. Also, from Theorem \ref{Theorem7.01} we have $b_{t+1,1}(n)\geq b_{t,1}(n)$, for all $t\geq2$ and $n\geq0$.
Therefore, for $k=1$, all the biases are known and the diagram is complete for all $t\geq2$ and $n\geq0$:
\begin{center}
\begin{tikzpicture}
\node (R1) at (0,0) {$b_{t+1,1}(n)$};
\node (R2) at (0,1.5) {$b_{t,1}(n)$};	
\node (R3) at (2.5,0) {$d_{t+1,1}(n)$};
\node (R4) at (2.5,1.5) {$d_{t,1}(n)$};
\node (R5) at (0,0.75) {$\rotatebox{90}{\ensuremath{\geqslant}}$};
\node (R6) at (1.25,0) {$\leqslant$};
\node (R6) at (1.25,1.5) {$\leqslant$};
\node (R6) at (2.5,0.75) {$\rotatebox{90}{\ensuremath{\geqslant}}$};
\end{tikzpicture}
\end{center}
It is known due to Ballantine et al. \cite{Ballantine_2023} that $b_{2,2}(n)\geq d_{2,2}(n)$ for all $n\geq0$ and $b_{2,3}(n)\geq d_{2,3}(n)$ for all $n\geq8$. Also, we have Theorems \ref{Thm_main7.1} and \ref{Thm_main7.2}. Therefore, for $k=2,3$, we have the following diagram for all but finitely many $n\geq 0$:
\begin{center}
\begin{tikzpicture}
\node (R1) at (0,0) {$b_{3,k}(n)$};
\node (R2) at (0,1.5) {$b_{2,k}(n)$};	
\node (R3) at (2.5,0) {$d_{3,k}(n)$};
\node (R4) at (2.5,1.5) {$d_{2,k}(n)$};
\node (R5) at (0,0.75) {$\rotatebox{90}{\ensuremath{\geqslant}}$};
\node (R6) at (1.25,0) {$?$};
\node (R6) at (1.25,1.5) {$\geqslant$};
\node (R6) at (2.5,0.75) {$\rotatebox{90}{\ensuremath{\geqslant}}$};
\end{tikzpicture}
\end{center}
Our method of the proof of Theorem \ref{Thm_main7.1} can not be generalized to prove the biases for the number of hooks of length 2 in $t$-regular partitions for the next values of $t$. However, numerical evidence suggest that the number of hooks of length 2 in $t$-regular partitions increases with increasing values of $t$. For example, in Table \ref{Table7.last} values in every column are in increasing order. In view of this, we propose the following conjecture.
\begin{conj}
Let $t\geq3$ be an integer. We have $b_{t+1,2}(n)\geq b_{t,2}(n)$, for all $n\geq0$.
\end{conj}
\begin{center}
	\begin{table}
		\caption{Values of $b_{t,2}(n): 1\leq n\leq 12$ and $3\leq t\leq13$}\label{Table7.last}
		\begin{tabular}[h]{c|cccccccccccc}
			$n\rightarrow$ & 1 & 2 & 3 & 4 & 5 & 6 & 7 & 8 & 9 & 10 & 11 & 12\\
			\hline
			$b_{3,2}(n)$ & 0 & 2 & 1 & 5 & 5 & 11 & 12 & 22 & 28 & 43 & 53 & 79\\
			$b_{4,2}(n)$ & 0 & 2 & 2 & 5 & 7 & 12 & 18 & 27 & 39 & 55 & 76 & 106\\
			$b_{5,2}(n)$ & 0 & 2 & 2 & 6 & 7 & 15 & 18 & 33 & 42 & 67 & 87 & 129\\
			$b_{6,2}(n)$ & 0 & 2 & 2 & 6 & 8 & 15 & 21 & 34 & 47 & 71 & 98 & 140\\
			$b_{7,2}(n)$ & 0 & 2 & 2 & 6 & 8 & 16 & 21 & 37 & 48 & 77 & 101 & 151\\
			$b_{8,2}(n)$ & 0 & 2 & 2 & 6 & 8 & 16 & 22 & 37 & 51 & 78 & 107 & 155\\
			$b_{9,2}(n)$ & 0 & 2 & 2 & 6 & 8 & 16 & 22 & 38 & 51 & 81 & 108 & 161\\
			$b_{10,2}(n)$ & 0 & 2 & 2 & 6 & 8 & 16 & 22 & 38 & 52 & 81 & 111 & 162\\
			$b_{11,2}(n)$ & 0 & 2 & 2 & 6 & 8 & 16 & 22 & 38 & 52 & 82 & 111 & 165\\
			$b_{12,2}(n)$ & 0 & 2 & 2 & 6 & 8 & 16 & 22 & 38 & 52 & 82 & 112 & 165\\
			$b_{13,2}(n)$ & 0 & 2 & 2 & 6 & 8 & 16 & 22 & 38 & 52 & 82 & 112 & 166
		\end{tabular}
	\end{table}
\end{center}  
Recently, several hook length biases among $t$-regular and $t$-distinct partitions have been established with the help of generating functions, see for example \cite{Ballantine_2023, Craig,Singh_Barman}. The generating functions of $b_{2,2}(n)$ and $b_{3,2}(n)$ are already known. Our proof of Theorem \ref{Thm_main7.1} does not use any generating function technique. It would be interesting to prove the bias established in Theorem \ref{Thm_main7.1} with the help of generating functions. To find a similar proof of Theorem \ref{Thm_main7.2} we need to first derive the generating function of $b_{3,3}(n)$ as it is not yet known. Further, it would be very interesting to know if for positive integers $k$ and $t\geq2$, there exists an integer $N_{t,k}$ such that $b_{t+1,k}(n)\geq b_{t,k}(n)$, for all $n\geq N_{t,k}$. This is true for  certain values of $t$ and $k$ as we see in Theorems \ref{Theorem7.01}, \ref{Thm_main7.1} and \ref{Thm_main7.2}. However, proving similar results for general values of $t$ and $k$ seems to be a hard problem.

\end{document}